\documentclass[leqno,11pt,a4paper]{amsart}

\usepackage{jc}

\usepackage{tikz}
\usepackage{tikz-cd}

\newcommand{\Aug}{\on{Aug}}

\newcommand{\CDGA}{\text{{\bf DA}}}

\begin{document}

\author{Julian Chaidez}
\address{Department of Mathematics\\University of Southern California\\Los Angeles, CA\\90007\\USA}
\email{julian.chaidez@usc.edu}

\title[Contact Homology And Linearization Without DGA Homotopies]{Contact Homology And Linearization Without DGA Homotopies}

\begin{abstract} This article clarifies the status of linearized contact homology given the foundations of the contact dg-algebra established by Pardon. In particular, we prove that the set of isomorphism classes of linearized contact homologies of a closed contact manifold is a contact invariant. 
\end{abstract}

\vspace*{-10pt}

\maketitle


\vspace*{-15pt}

\section{Introduction} \label{sec:introduction} 

Linearized contact homology is a flavor of contact homology associated to a closed contact manifold and an augmentation of its contact dg-algebra. It is computed from a free chain complex generated by closed Reeb orbits whose differential is acquired by a linearization procedure applied to the differential in the contact dg-algebra. Linearized contact homology is closely related to $S^1$-equivariant symplectic homology (cf. Bourgeois-Oancea \cite{bourgeois2017s}) and it has a number of applications to Reeb dynamics (cf. Colin-Honda \cite{colin2013reeb}). 

\vspace{3pt}

Contact homology, and more generally symplectic field theory (SFT), was first articulated by Eliashberg-Givental-Hofer in the seminal paper \cite{eliashberg2010introduction}, where details such as compactness and transversality were deferred to later works. SFT compactness was proven quickly following \cite{eliashberg2010introduction} by Bourgeois-Eliashberg-Hofer-Wysocki-Zehnder \cite{bourgeois2003compactness}. However, transversality in SFT requires non-classical methods (e.g. Kuranishi charts \cite{bao2023semi} or polyfolds \cite{fish2018lectures}) and remained unaddressed until Pardon \cite{pardon2019contact} resolved the genus zero (contact homology) case using the VFC methods of \cite{pardon2016algebraic}  (also see Bao-Honda \cite{bao2023semi}). Here is an informal and simplified version of Pardon's theorem.

\begin{theorem*}[Pardon] \label{thm:main_pardon} The contact dg-algebra $A(Y,\xi)$ of a closed contact manifold $(Y,\xi)$ is well defined up to dg-algebra quasi-isomorphism, canonical up to chain homotopy. Thus the full contact homology
\[CH(Y,\xi) = H(A(Y,\xi),\partial) \qquad\text{is well-defined as a $\Z/2$-graded algebra}\]
\end{theorem*}
\noindent Pardon also established functoriality of full contact homology with respect to exact cobordisms. 

\vspace{3pt}

Although \cite{pardon2019contact} resolved the basic issue of the well-definedness of full contact homology, some foundational questions remain open. In particular, \cite{pardon2019contact} only proved that the dg-algebra maps induced by exact cobordisms are well-defined up to chain homotopy, and not up to the stronger relation of dg-algebra homotopy. This is a particular problem for linearized contact homology. For example, the foundations of \cite{pardon2019contact} do not allow one to associate a linearized contact homology group $LCH(W)$ (well-defined up to canonical isomorphism) to an exact filling $W$.  

\vspace{3pt}

The purpose of this short note is to clarify that a version of well-definedness for linearized contact homology follows from only the foundations of \cite{pardon2019contact} and some basic homological algebra. 

\begin{theorem*}[Main Theorem] \label{thm:main_theorem} The set $\on{Aug}(Y,\xi)$ of weak equivalence classes of augmentations of a closed contact manifold $(Y,\xi)$ is well-defined up to canonical bijection, and the linearized contact homology
\[LCH_{[\epsilon]}(Y,\xi) \qquad\text{of an augmentation class $[\epsilon] \in \on{Aug}(Y,\xi)$}\] is well-defined as an $\Z/2$-graded vectorspace over $\Q$ up to (non-canonical) isomorphism.
\end{theorem*}
\noindent This result can also be modified to account for other gradings and coefficients (see Remark \ref{rmk:gradings_coeffs}). Following Pardon \cite{pardon2019contact}, Theorem \ref{thm:main_theorem} was mentioned by several authors as an unresolved problem, e.g. by Pardon himself \cite[p. 14]{pardon2019contact}, Bao-Honda \cite[Warn 1.05(b)]{bao2023semi}, Moreno-Zhou \cite[Rmk 4.4]{moreno2020landscape} and Hind-Siegel \cite[p. 21]{hind2024symplectic}. We hope that this note will clarify the status of this problem.


\vspace{3pt}

Statements and proofs of our main invariance results appear in Section \ref{sec:contact_homology} after a brief discussion of the preliminary homological algebra in Section \ref{sec:homological_algebra}. In Section \ref{sec:model_categories}, we discuss model categories and some additional useful homological results. We apply these in Section \ref{sec:example_applications} to prove a number of folklore uniqueness results for the linearized contact homology of SADC contact manifolds \cite{zhou2021symplectic}.  

\newpage

\section{Homological Algebra} \label{sec:homological_algebra} We first discuss the homological algebra needed to prove the main results. Our aim is to give an elementary and self-contained treatment of this material. 

\subsection{Pointed DG-Algebras} We start with a brief discussion of dg-algebras including the special class, called Sullivan algebras, that will be of primary interest. For the rest of the section, we fix
\[\text{a coefficient ring $R$ over $\Q$} \qquad\text{and}\qquad \text{a grading group $\Z/2m$ for $m \in \N$}\]
\begin{remark}[Conventions] In this section, we adopt cohomological grading conventions. All modules and algebras have coefficients in $R$ and are graded by $\Z/2m$ unless otherwise specified. \end{remark}
\begin{definition} A \emph{pointed commutative differential graded algebra} $(A,\epsilon)$ is a pair of a graded-commutative differential graded algebra $A$ and a unital map of dg-algebras
\[\epsilon:A \to R \qquad\text{to the coefficient ring $R$ with trivial differential and grading}\]
A map of pointed commutative differential graded algebra is a map of dg-algebras 
\[
 \Phi:(A,\epsilon) \to (B,\mu) \qquad\text{such that}\qquad \mu \circ \Phi = \epsilon
\]
The pair $(A,\epsilon)$ is alternatively called a \emph{pointed cdg-algebra}. The map $\epsilon$ is called an \emph{augmentation} and we denote the kernel of the augmentation by
\[A_\star = \on{ker}(\epsilon)\]\end{definition}

\begin{definition} A \emph{weak equivalence} $\Phi:A \simeq B$ of pointed commutative differential graded algebras is a cdga map that is a quasi-isomorphism, i.e. such that the induced map on homology
\[H\Phi:HA \to HB \qquad\text{is an isomorphism}\]
\end{definition}

\begin{definition}\label{def:Sullivan} A cdg-algebra $A$ is \emph{Sullivan} if there is a graded algebra isomorphism
\[A \simeq SV \qquad\text{with the free graded-commutative algebra $SV$ over a free $R$-module $V$}\]
where $V$ is equipped with free $R$ modules $M_i$ indexed by a well-ordered set $I$ such that
\begin{equation} \label{eq:sullivan_filtration} \bigoplus_{i \in I} M_i \to V \text{ is an isomorphism}\qquad\text{and}\qquad dM_i \subset S\Big(\bigoplus_{j < i} M_j\Big)\end{equation}
We let $V_i \subset V$ denote the direct sum of all sub-modules $M_j \subset V$ with $j \le i$ for a fixed $i \in I$. \end{definition}

\begin{example} The \emph{interval algebra} $(P,d)$ is the cdg-algebra with $P = SU$ where $U$ be the free $R$-module with a generators $s$ in degree $0$ and $t$ in degree $1$, and differential defined by $ds = t$.
\end{example}

We will need a notion of homotopy that is compatible with the dg-algebra structure. There are several different possible definitions, but we will use the following one.

\begin{definition} The \emph{path cdg-algebra} $(PA,P\epsilon)$ of a pointed cdg-algebra $(A,\epsilon)$ is the pointed cdg-algebra defined as follows. The graded unital algebra is given by
\[
PA = R \oplus (A_\star \otimes P) \qquad\text{where $P$ is the interval cdg-algebra}
\]
The differential and product on $PA$ restrict to the standard (graded-commutative) tensor product differential and product on $A_\star \otimes P \subset PA$, and $1 \in R \subset PA$ is the unit of $PA$ which is closed. These properties determine the dg-algebra structure. The augmentation $P\epsilon$ is given by projection
\[
P\epsilon:PA \to R \qquad\text{with}\qquad P\epsilon(r \oplus z) = r \qquad\text{for}\qquad r \oplus z \in R \oplus (A_\star \otimes P)
\]
Finally, there are natural pointed cdg-algebra maps $\Pi_i:PA \to A $ for $i = 0,1$ given by
\[
\Pi_i(r + \sum_{m=0}^\infty x_m \otimes s^m + \sum_{n=0}^\infty y_n \otimes s^n ds) = r + \sum_{m=0}^\infty i^m \cdot x_m  \in R \oplus A_\star \simeq A
\]
In this formula, one should take $s^m = 1$ when $m = 0$ and $i^m = 0$ when $i = 0$ and $m = 0$. \end{definition}

\begin{definition} A \emph{cdga-homotopy} $H:\Phi \simeq \Psi$ between pointed cdg-algebra maps $\Phi$ and $\Psi$ from a pointed cdg-algebra $(A,\epsilon)$ to a pointed cdg-algebra 
 $(B,\mu)$ is a map
 \[
 H:A \to PB \qquad \text{such that}\qquad \Pi_0 \circ H = \Phi \text{ and }\Pi_1 \circ H = \Psi
 \]
\end{definition}

\begin{remark}[Characteristic Zero] Note that the inclusion $R \to P$ is a homotopy equivalence of chain complexes if and only if $R$ is characteristic zero. This partly motivates the requirement that $R$ be characteristic zero (aside from its use in Lemma \ref{lem:main_homological} below).
\end{remark}

\subsection{Main Homological Lemma.} We can now state and prove the main result in homological algebra that we will need for the rest of the paper.

\begin{lemma}[Homotopy Invertibility] \label{lem:main_homological} Let $\Phi:(A,\epsilon) \to (B,\mu)$ be a weak equivalence of pointed cdg-algebras such that $B$ is Sullivan. Then there is a weak-equivalence
\[
\Psi:(B,\mu) \to (A,\epsilon) \qquad\text{with}\qquad \Phi \circ \Psi \simeq \on{Id}_B
\]
\end{lemma}

\begin{proof} By Definition \ref{def:Sullivan}, we may take $B = SV$ where $V$ is a free graded $R$-module equipped with free sub-modules $M_i \subset V$ indexed by a well-ordered set $I$ satisfying (\ref{eq:sullivan_filtration}). Consider the map
\[
\iota:SV \to SV \qquad\text{defined by}\qquad \iota(v) = v - \epsilon(v) \text{ for each $v \in V$}
\]
The pullback of the augmentation $\mu$ by $\iota$ has $V$ in its kernel and the sub-modules $M_i$ still satisfies (\ref{eq:sullivan_filtration}) with respect to the pullback of $d$ by $\iota$. Thus we may assume without loss of generality that
\[
V \subset \on{ker}(\mu) = B_\star
\]
We now use induction on $i \in I$ to construct a cdg-algebra map and a homotopy of the form
\[
\Psi:(SV_i,\mu) \to (A,\epsilon) \quad\text{and}\quad H:(SV_i,\mu) \to (P(SV_i),P\mu) \quad\text{with}\quad \Pi_0 \circ H = \on{Id} \text{ and }\Pi_1 \circ H = \Phi \circ \Psi
\]

{\bf Base Case.} For the base case, let $1 \in I$ denote the minial element of $I$ and fix a basis $S_1 \subset M_1$ of $M_1$ as a free module. Since $1$ is the minimal element of $I$, Definition \ref{def:Sullivan} implies that
\[
d(M_1) = 0 \qquad\text{and thus}\qquad dv = 0 \qquad\text{for each element $v \in S_1$}
\]
Since $\Phi$ is a quasi-isomorphism, there is a closed $a \in A$ with $\Phi(a) = v + db$ for some $b \in B$. We may choose $b \in B_\star = \on{ker}(\mu)$ since $B = R \oplus B_\star$ and $d$ is zero on $R \subset B$. We define $\Psi$ and $H$ by
\[
\Psi(v) = a \qquad\text{and}\qquad H(v) = v + d(b \otimes s) \in B_\star \otimes P \subset PB \qquad\text{on each basis element $v \in S_1$}
\]
To check that $\Psi$ and $H$ commute with the differential on $SV_1$, we simply note that
\[
\Psi(dv) = 0 = da = d\Psi(v) \quad\text{and}\quad H(dv) = 0 = d(v_1 + d(b \otimes s)) = dH(v)
\]
To check that $H$ is a homotopy from $\on{Id}_B$ to $\Phi \circ \Psi$ restricted to $SV_1$, we simply note that
\[
\Pi_0 \circ H(v) = v \quad\text{and}\quad \Pi_1 \circ H(v) = v + db = \Phi \circ \Psi(v)
\]

{\bf Induction Step.} Next, suppose that $\Psi$ and $H$ have been extended to $SV_j$ for all $j < i$ such that
\[
\Pi_0 \circ H = \on{Id}_B \qquad\text{and}\qquad \Pi_1 \circ H = \Phi \circ \Psi \qquad\text{restricted to $SV_j$ for all $j < i$}
\]
Fix a basis $S_i \subset M_i$ of the free $R$-module $S_i$ and fix $v \in S_i$. Note that $dM_i \subset S(\bigoplus_{j < i} M_j)$ by the Sullivan property (\ref{eq:sullivan_filtration}). Thus $H(dv)$ is well defined and $\Pi_0 \circ H(dv) = dv$ by induction. Therefore
\[
H(dv) = dv + \sum_{j=1}^\infty x_j \otimes s^j + \sum_{k=0}^\infty y_k \otimes s^k ds
\]
Here $x_j$ and $y_k$ are elements of $B_\star$, and all but finitely many of them are zero. Since $H(dv)$ is closed, we must have that
\[
dH(dv) = \sum_j dx_j \otimes s^j + \sum_k ((-1)^{|x_{k+1}|} \cdot (k+1) \cdot x_{k+1} + dy_k) \otimes s^k ds = 0
\]
It follows that $dy_{k-1} = -(-1)^{|x_k|} \cdot k \cdot x_k$. We now use the assumption that $R$ is a ring over $\Q$, so that all integers $k \ge 1$ are invertible. This permits us to write
\[
H(dv) = dZ \qquad\text{where}\qquad Z = v + \sum_k (-1)^{|y_k|} \cdot \frac{1}{k} \cdot y_{k-1} \otimes s^k
\]
Let $z = \Pi_1(Z)$. Since $H$ is a homotopy from $\on{Id}_B$ to $\Phi \circ \Psi$ on $S(\oplus_{j < i} M_j)$, we know that
\[
dz = \Pi_1(dZ) = \Pi_1 \circ H(dv) = \Phi \circ \Psi(dv)
\]
Since $\Phi$ is a weak equivalence and $\Psi(dv)$ is a cycle by the inductive construction of $\Psi$, we may choose a $w$ such that
\[
dw = \Psi(dv) \quad\text{and thus}\quad d(z - \Phi(w)) = dz - \Phi(dw) = dz - \Phi \circ \Psi(dv_i) = 0
\]
Then since $z - \Phi(w)$ is closed and $\Phi$ is a weak equivalence, we can pick a closed element $c \in A$ and an element $b \in B_\star$ with
\[
\Phi(c) = z - \Phi(w) + db \qquad\text{or equivalencely}\qquad \Phi(c + w) = z + db
\]
Finally, we can define $\Psi$ and $H$ on the basis element $v$ in terms of $c,w,Z$ and $b$ as follows.
\[
\Psi(v) = c + w \qquad\text{and}\qquad H(v) = Z + d(b \otimes s)
\]
To check that these $\Psi$ and $H$ commute with the differential, we note that
\[
d\Psi(v) = dw = \Psi(dv) \qquad\text{and}\qquad dH(v) = dZ = H(dv)
\]
To check the homotopy property, we simply note that
\[
\Pi_0 \circ H(v) = \Pi_0(Z) = v \qquad\text{and}\qquad \Pi_1 \circ H(v) = \Pi_1(Z + d(b \otimes s)) = z + db = \Phi \circ \Psi(v)
\]
By using this definition for every basis element $v \in S$ and extending to an algebra map, we acquire the desired extension of $\Psi$ and $\Phi$ to $SV_i$. This completes the induction and the lemma. \end{proof}

\subsection{Linearization} We next review the construction of the linearization of a pointed cdg-algebra. 

\begin{definition} The \emph{linearized complex} $LC(A,\epsilon)$ of a pointed cdg-algebra $(A,\epsilon)$ is the complex
\[
LC(A,\epsilon) = A_\star/A_\star^2
\]
with differential induced by the differential of $A$. The \emph{linearized chain map}
\[
L\Phi:LC(A,\epsilon) \to LC(B,\mu) \qquad\text{of a pointed cdg-algebra map }\Phi:(A,\epsilon) \to (B,\mu)
\]
is the induced map on the corresponding quotients. Finally, the \emph{linearized homology} $LH(A,\epsilon)$ of a cdg-algebra $(A,\epsilon)$ is the homology of $LC(A,\epsilon)$. \end{definition}

\begin{remark} In the homological algebra literature, the linearized complex is commonly referred to as the indecomposables. Our language is more common in the contact topology literature.
\end{remark}

\begin{example}[Sullivan Linearization] \label{ex:Sullivan_linearization} Fix a Sullivan cdg-algebra $(SV,\partial)$ with an augmentation $\epsilon$. Let $\iota:SV \to SV$ be the unique graded algebra map defined on $V$ by $\iota(v) = v - \epsilon(v)$. Then we have $\epsilon \circ \iota|_V = 0$ and there is an isomorphism
\[
V \simeq \on{ker}(\epsilon \circ \iota)/\on{ker}(\epsilon \circ \iota)^2 \simeq LC(SV,\epsilon) \qquad\text{induced by the inclusion }V \subset \on{ker}(\epsilon \circ \iota)
\]
The pullback differential $\partial_\epsilon = \iota^*\partial$ is non-decreasing with respect to the word filtration of $SV$ and the differential on $V$ is the word length one part of $\partial_\epsilon|_V$ under the isomorphism with $LC(SV,\epsilon)$. \end{example}

We will require the following elementary results about linearized homology.

\begin{lemma}[Path Object] \label{lem:path_object_linearization} For any pointed cdg-algebra $(A,\epsilon)$, there is a canonical isomorphism
\[LH(PA,P\epsilon) = LH(A,\epsilon) \qquad\text{induced by the linearizations $L\Pi_i$ for $i = 0,1$}\]
\end{lemma}

\begin{proof} The linearized complex of $(PA,P\epsilon)$ is simply give by
\[
(A_\star \otimes P)/(A_\star \otimes P)^2 \xrightarrow{=} (A_\star/A_\star^2) \otimes P \simeq LC(A,\epsilon) \otimes P
\]
The map $\iota:LC(A,\epsilon) \otimes R \to LC(A,\epsilon) \otimes P$ induces a canonical isomorphism, since $R$ is characteristic zero. The linearized maps $L\Pi_i$ are both equal to the inverse map to $\iota$ on linearized homology. Indeed $L\Pi_i$ is given on chain level by the map
\[
LC(A,\epsilon) \otimes P \to LC(A,\epsilon) \otimes R = LC(A,\epsilon) \qquad\text{with}\qquad x \otimes p \mapsto x \otimes \on{ev}_i(p)
\]
Here $\on{ev}_i:P \to R$ is the map given by
\[
\on{ev}_i(p) = r + \sum_m a_m \cdot i^m \qquad\text{if}\qquad p = r + \sum_m a_m \cdot s^m + \sum_n b_n \cdot s^n ds
\]
In particular, $L\Pi_i \circ \iota = \on{Id}$ and so $L\Pi_i$ is the inverse of $\iota$ on homology for both $i = 0$ and $i = 1$. \end{proof}

\begin{lemma}[Homotopy] \label{lem:homotopy_linearization}Let $\Phi,\Psi:(A,\epsilon) \to (B,\mu)$ be homotopic maps of pointed cdg-algebras. Then
\[
L\Phi = L\Psi \qquad\text{as maps}\qquad LH(A,\epsilon) \to LH(B,\mu)
\]
\end{lemma}

\begin{proof} Let $H:(A,\epsilon) \to (PB,P\mu)$ be the homotopy $\Phi \simeq \Psi$. Then by Lemma \ref{lem:path_object_linearization}, we have
\[L\Phi = L\Pi_0 \circ LH = L\Pi_1 \circ LH = L\Psi \qquad\text{on homology} \qedhere\]
\end{proof}

\begin{lemma}[Sullivan] \label{lem:sullivan_linearization} Let $\Phi:(A,\epsilon) \to (B,\mu)$ be a morphism of pointed Sullivan cdg-algebras. Then $\Phi$ is a weak equivalence if and only if the induced map $L\Phi$ on the linearized complex if a quasi-isomorphism.
\end{lemma}

\begin{proof} If $\Phi$ is a weak equivalence, then by Lemma \ref{lem:main_homological}, there is a weak equivalence $\Psi:(B,\mu) \to (A,\epsilon)$ such that $\Phi \circ \Psi \simeq \on{Id}_B$. By Lemma \ref{lem:homotopy_linearization}, this implies that $L\Phi \circ L\Psi = \on{Id}$ on homology. Thus $L\Phi$ is surjective and $L\Psi$ is injective on linearized homology. The same argument applied to $\Psi$ implies that $L\Psi$ is onto on homology. Thus $L\Phi$ is invertible on homology with inverse $L\Psi$.

\vspace{3pt}

Conversely, suppose that $L\Phi$ is a quasi-isomorphism. Choose identifications $SU \simeq A$ and $SV \simeq B$ for graded vector spaces $U$ and $V$. Following Example \ref{ex:Sullivan_linearization}, we may assume that the differentials $SU$ and $SV$ are non-decreasing for the word length filtration. Then by Example \ref{ex:Sullivan_linearization} the map on homology $HU \to HV$ is an isomorphism where $U$ and $V$ are equipped with the differential given by the word length one part of the differentials on $A$ and $B$. Since $R$ is characteristic zero, this implies that the map
\[
H(\on{Gr} A) = S(HU) \to S(HV) = H(\on{Gr} B)
\]
is an isomorphism, where $\on{Gr} A$ and $\on{Gr} B$ are the associated graded complexes with respect to the word filtration. A standard argument with the spectral sequence associated to the word length filtration then implies that the map $HA \to HB$ is an isomorphism (cf. Pardon \cite[Lem 1.2]{pardon2019contact}). \end{proof}

\subsection{Augmentations} We conclude this section by shifting perspective, to considering augmentations on a fixed cdg-algebra.  

\begin{definition} \label{def:weak_equivalence_of_augs} A \emph{weak equivalence} $\epsilon \simeq \mu$ of augmentations $\epsilon$ and $\mu$ of a cdg-algebra $A$ is a weak equivalence of pointed cdg-algebra
\[(A,\epsilon) \to (A,\mu)\]\end{definition}

\begin{lemma} Weak equivalence is an equivalence relation on augmentations if $A$ is Sullivan.
\end{lemma}

\begin{proof} Reflexivity and transitivity are trivial. Symmetry is immediate from Lemma \ref{lem:main_homological}.
\end{proof}

\begin{definition} The set of \emph{augmentation classes} $\Aug(A)$ of a Sullivan cdg-algebra $A$ is the set of augmentations $\epsilon:A \to R$ modulo weak equivalence of augmentations.
\end{definition}

\begin{lemma} \label{lem:weak_equivalent_Aug} Let $A$ and $B$ be weakly equivalent Sullivan algebras. Then there is a canonical bijection
\[\Aug(A) = \Aug(B)\]
\end{lemma}

\begin{proof} Choose a weak equivalence $\Phi:A \to B$ and consider the map on augmentations 
\begin{equation} \label{eq:aug_map}
\Phi^*:\on{Aug}(B) \to \on{Aug}(A) \qquad \text{given by}\qquad \Phi^*[\epsilon] = [\epsilon \circ \Phi]
\end{equation}
First, note that the map (\ref{eq:aug_map}) is well-defined. Indeed, if $\mu \simeq \epsilon$, then there is a weak equivalence $\Psi$ of $B$ with $\mu = \epsilon \circ \Psi$. Moreover, by Lemma \ref{lem:main_homological}, there is a weak equivalence $\Phi':B \to A$ such that $\epsilon \circ \Phi \circ \Phi' = \epsilon$. Therefore we may write
\[\mu \circ \Phi = \epsilon \circ \Psi \circ \Phi =  \epsilon \circ \Phi \circ (\Phi' \circ \Psi \circ \Phi)\]
The paranthesized term is a weak equivalence of $A$ and so $\mu \circ \Phi \simeq \epsilon \circ \Phi$ in $\Aug(A)$. Next, note that the map (\ref{eq:aug_map}) is independent of $\Phi$. If $\Psi:A \to B$ is a different weak equivalence, we write
\[
\epsilon \circ \Psi = \epsilon \circ \Phi \circ (\Phi' \circ \Psi)
\]
where the parenthesized term is again a weak equivalence of $A$, so that $\epsilon \circ \Psi \simeq \epsilon \circ \Phi$. Finally, the map (\ref{eq:aug_map}) is invertible. Indeed, the inverse is given by
\[
\Psi^*:\on{Aug}(A) \to \on{Aug}(B) \qquad \text{given by}\qquad \Psi^*[\epsilon] = [\epsilon \circ \Psi]
\]
for any weak equivalence $\Psi:B \to A$. This proves that (\ref{eq:aug_map}) is a canonical bijection.  \end{proof}

Finally, we have the following result that is an immediate consequence of Lemma \ref{lem:sullivan_linearization}.

\begin{lemma} \label{lem:linearized_homology_class} Let $A$ be a weak equivalence class of Sullivan cdg-algebras. Then there is a canonically associated set of augmentation classes $\on{Aug}(A)$ and an isomorphism class of graded module
\[LH(A,[\epsilon]) \qquad\text{for each}\qquad [\epsilon] \in \Aug(A,[\epsilon])\]
\end{lemma}

\section{Contact Homology} \label{sec:contact_homology}

We now apply the main homological lemma (Lemma \ref{lem:main_homological}) to the contact dg-algebra to prove  the main invariance result (Theorem \ref{thm:main_theorem})   from the introduction.

\subsection{Contact DG-Algebra} We start by briefly reviewing the construction of the contact dg-algebra as presented by Pardon \cite{pardon2019contact}. Fix a closed contact manifold
\[(Y,\xi) \qquad\text{of dimension $2n-1$}\]

{\bf Preliminaries.} We start with some preliminary terminology on Reeb orbits. Fix a periodic (and possibly multiply covered) Reeb orbit $\Gamma$ of a contact form $\alpha$ of period $L$. Recall that $\Gamma$ is \emph{non-degenerate} if the linearized return map of the Reeb flow $\Phi$ restricted to $\xi$ satisfies
\[
T\Phi_L|_\xi:\xi_P \to \xi_P\qquad \text{has no $1$-eigenvalues for every $P \in \Gamma$}
\]
Every non-degenerate Reeb orbit $\Gamma$ has a well-defined mod $2$ Conley-Zehnder index and mod $2$ SFT grading \cite[Def 2.48]{pardon2019contact} given by
\[
\CZ(\Gamma) = \CZ(\Gamma,\tau) \in \Z/2 \qquad\text{and}\qquad |\Gamma| = n - 3 + \CZ(\Gamma) \in \Z/2
\]
To each basepoint $P \in \Gamma$, one can assign a rank one graded $\Z$-module, the \emph{orientation line}
\[
\mathfrak{o}_{\Gamma,P} \qquad\text{concentrated in grading }|\Gamma|
\]
The orientation lines naturally form a rank one local system on $\Gamma$ and $\Gamma$ is \emph{good} if this local system is trivial \cite[Def 2.49]{pardon2019contact}. Thus any good orbit has a natural orientation line $\mathfrak{o}_\Gamma$ independent of $P$. Finally, a contact form $\alpha$ is called \emph{non-degenerate} if every closed Reeb orbit is non-degenerate.

\vspace{3pt}

{\bf Basic Definition.} The contact dg-algebra of $(Y,\xi)$ is a differential graded algebra denoted
\[
A(Y,\alpha) \qquad\text{with differential} \qquad \partial_{J,\theta}:A(Y,\alpha) \to A(Y,\alpha)
\]
associated to choices of a non-degenerate contact form $\alpha$ on $(Y,\xi)$, a compatible almost-complex structure $J$ on $\xi$ \cite[p. 3-4]{pardon2019contact} and a perturbation datum $\theta \in \Theta(Y,\alpha,J)$ \cite[p. 5]{pardon2019contact}. A triple $(\alpha,J,\theta)$ will be referred to as \emph{Floer data}. The algebra $A(Y,\alpha)$ is the free graded-commutative algebra generated by the direct sum of the orientation lines (tensored with $\Q$) over all good orbits.
\[
V(Y,\alpha) = \bigoplus_{\Gamma\text{ good}} \mathfrak{o}_{\Gamma} \otimes \Q  \qquad\text{and}\qquad A(Y,\alpha) = SV(Y,\alpha) 
\]
The differential $\partial_{J,\theta}$ is constructed by counting points in compactified moduli spaces of pseudo-holomorphic curves of zero virtual dimension in the symplectization $\R \times Y$ \cite[\S 2.3 and 2.10]{pardon2019contact}. 

\vspace{3pt}

{\bf Homology Grading.} The contact dg-algebra admits a dg-algebra grading by the first homology group of $Y$, denoted by
\[
A(Y,\alpha) = \bigoplus_{Z \in H_1(Y)} A_Z(Y,\alpha)
\]
Here $A_Z(Y,\alpha)$ is spanned by monomials $x_1 \dots x_k$ where $x_i \in \mathfrak{o}_{\Gamma_i}$ and $Z = [\Gamma_1] + \dots + [\Gamma_k]$. This grading is respected by the cobordism maps in contact homology, and so there is a decomposition
\[
A(Y,\xi) =  \bigoplus_{Z \in H_1(Y)} A_Z(Y,\xi) \qquad\text{well defined up to quasi-isomorphism}
\]
The component $A_0(Y,\xi) \subset A(Y,\xi)$ is a differential graded sub-algebra that admits a $\Z/2m_\xi$-grading refining the $\Z/2$-grading, where $m_\xi$ is the divisibility of the first Chern class.
    \[m_\xi := \on{min}\{ c_1(\xi) \cdot A \; : \; A \in H_2(Y;\Z)\}\]

\vspace{3pt}

{\bf Action Filtration.} The contact dg-algebra $A(Y,\alpha)$ has a natural $\R$-filtration by dg-subalgebras
\[A^L(Y,\alpha) \subset A(Y,\alpha)\]
Precisely, the vectorspace $V(Y,\alpha)$ has an $\R$-filtration via the periods of the closed orbits.
\[
V_L(Y,\alpha) = \bigoplus_{\Gamma \in \mathcal{P}(L)} \mathfrak{o}_\Gamma \otimes \Q \qquad\text{where $\mathcal{P}(L)$ is the set of good orbits of $\alpha$ with period $L$ or less}
\]
Since the contact form $\alpha$ is non-degenerate, the set of closed Reeb orbits of period $L$ or less is finite. It follows that $V_L(Y,\alpha)$ is a proper and complete $\R$-filtration on $V(Y,\alpha)$, meaning that
\[
V_L(Y,\alpha)\text{ is finite dimensional for any $L$} \qquad\text{and}\qquad V(Y,\alpha) = \underset{L}{\on{colim}} V(Y,\alpha)
\]
Moreover, the differential $\partial_{J,\theta}$ decreases the filtration (cf. \cite[\S 1.8, Bullet 3]{pardon2019contact}) in the sense that
\[
\partial_{J,\theta}(V_L(Y,\alpha)) \subset SV_K(Y,\alpha) \qquad\text{for some }K < L
\]
The action filtration is then defined by $A^L(Y,\alpha) = SV_L(Y,\alpha)$. In the language of Section \ref{sec:homological_algebra}, the existence of the action filtration implies that the contact dg-algebras are Sullivan.

\begin{lemma} \label{lem:contact_dga_is_Sullivan} The contact dg-algebra $(A(Y,\alpha),\partial_{J,\theta})$ is Sullivan for any choice of Floer data $(\alpha,J,\theta)$.
\end{lemma}

\begin{proof} We can order the orbits $\Gamma_1, \Gamma_2 \dots$ so that the period of $\Gamma_i$ is less than or equal to the period of $\Gamma_{i+1}$. Then we let $V(Y,\alpha)$ be defined as above and let $M_i \subset V(Y,\alpha)$ be the subspace $\mathfrak{o}_{\Gamma_i} \otimes \Q$. Then $A(Y,\alpha) = SV(Y,\alpha)$ and the subspaces $M_i$ indexed by $I = \N$ satisfy Definition \ref{def:Sullivan}. \end{proof}

\subsection{Linearized Contact Homology} We can now prove the results from the introduction. By Theorem \ref{thm:main_pardon} and Lemma \ref{lem:contact_dga_is_Sullivan}, there is a well-defined equivalence class of Sullivan cdg-algebra
\[A(Y,\xi)\]
represented by the Sullivan cdg-algebra $A(Y,\alpha)$ for any choice of Floer data $(\alpha,J,\theta)$. By Lemma \ref{lem:weak_equivalent_Aug}, this equivalence class has a canonically associated set of augmentation classes. Moreover, by Lemma \ref{lem:linearized_homology_class}, each augmentation class has an associated linearized homology, well-defined up to non-canonical isomorphism. Thus we may make the following definitions.

\begin{definition} \label{def:aug_set_contact} The set of \emph{augmentation classes} of a closed contact manifold $(Y,\xi)$ is given by
\[
\on{Aug}(Y,\xi) = \on{Aug}(A(Y,\xi))
\]
\end{definition}

\begin{definition} \label{def:LCH} The \emph{linearized contact homology} $LCH_{[\epsilon]}(Y,\xi)$ of a closed contact manifold $(Y,\xi)$ and an augmentation class $[\epsilon] \in \on{Aug}(Y,\xi)$ is the (isomorphism class of) graded module
\[
LCH_{[\epsilon]}(Y,\xi) = LH(A(Y,\xi),[\epsilon])
\]
 \end{definition}

\noindent The well-definedness of augmentation classes and linearized contact homology groups together constitute the main theorem (Theorem \ref{thm:main_theorem}) from the introduction.

\begin{remark}[Gradings/Coefficients] \label{rmk:gradings_coeffs}There are many circumstances where the gradings and coefficients of the dg-algebra $A(Y,\xi)$ can be refined. Of particular note are the following cases.
\begin{itemize}
    \item If $H_1(Y) = 0$ then the $\Z/2$-grading can be enhanced to a $\Z/2m_\xi$-grading where
    \[m_\xi = \on{min}\{ c_1(\xi) \cdot A \; : \; A \in H_2(Y;\Z)\}\]
    \item Following \cite{eliashberg2010introduction}, one may define $A(Y,\xi)$ to have coefficients in the group ring of $H_2(Y)$ over $\Q$ graded by $|A| = 2c_1(\xi) \cdot A$ and the resulting dg-algebra has a $\Z$-grading.
\end{itemize}
In both of these cases, the coefficient group $R$ and grading $\Z/2m$ satisfy the assumptions of Section \ref{sec:homological_algebra}. Thus the appropriately modified version of Theorem \ref{thm:main_theorem} and Definitions \ref{def:aug_set_contact}-\ref{def:LCH} apply.\end{remark}

There are a number of stronger invariance results that cannot be obtained with the simple algebraic tools in Section \ref{sec:homological_algebra}. We briefly discuss these results.

\begin{remark}[Fillings] \label{rmk:fillings} Theorem \ref{thm:main_theorem} does not address the well-definedness of the linearized contact homology associated to an exact filling $W$. Morally, this group should be defined by
\[
LCH_{[\epsilon_W]}(Y,\xi) \qquad\text{where}\qquad \epsilon_W:A(Y,\xi) \to A(\emptyset) = \Q \text{ is the contact dga cobordism map}
\]
Unfortunately, the foundations of Pardon \cite{pardon2019contact} only establish that $\epsilon_W$ is well-defined up to chain homotopy, which is an relation that is not clearly related to our notion of weak equivalence. \end{remark}

\begin{remark}[Weak Functoriality] \label{rmk:weak_functoriality} A simple version of functoriality of linearized contact homology would state that an exact cobordism $X:(Y,\xi) \to (Z,\eta)$ induces maps
\[
\Phi^*_X:\on{Aug}(Z,\xi) \to \on{Aug}(Y,\eta) \qquad\text{and}\qquad LCH_X:LCH(Y,\Phi^*_X[\epsilon]) \to LCH(Z,[\epsilon])
\]
where the latter may be interpreted as being well-defined up to left and right action by automorphisms of the domain and target. Such maps can be defined for a given choice of Floer data, but they a priori depend on these choices, again due to the well-definedness of cobordism maps only up to chain homotopy of maps between contact dg-algebras. \end{remark}

\section{Model Categories And CDGAs} \label{sec:model_categories} The homological algebra developed in Section \ref{sec:homological_algebra} is best understood via the formalism of model categories. Here we provide a brief overview of model categories, to both contextualize Section \ref{sec:homological_algebra} and to recall some useful results for the applications in Section \ref{sec:example_applications}. 

\subsection{Overview} Model category theory, due to Quillen \cite{quillen2006homotopical}, is an abstract categorical framework for homotopy theory, and is closely related to more modern formalisms (e.g. $\infty$-categories \cite{lurie2009higher}). 

\vspace{3pt} Roughly, a \emph{model category} $\mathcal{C}$ is a complete and cocomplete category equipped with three distinguished classes of morphisms
\[\text{weak equivalences $\on{Equiv}(\mathcal{C})$}\qquad \text{fibrations $\on{Fib}(\mathcal{C})$} \qquad\text{cofibrations $\on{Cof}(\mathcal{C})$}\]
The weak equivalences must contain isomorphisms and satisfy the \emph{two-out-of-three} property, and the three classes satisfy several factorization axioms. See Quillen \cite{quillen2006homotopical} or Hovey \cite{hovey2007model} for a detailed definition. An object $A$ is called \emph{cofibrant} if the map from the initial object is a cofibration, and \emph{fibrant} if the map to the final object is a fibration. 

\vspace{3pt}

Model categories provide a framework for formulating homotopy categories and derived functors. There is a  natural notion of homotopy between morphisms, and the \emph{homotopy category}
\[
\on{Ho} \mathcal{C} \qquad\text{of a model category $\mathcal{C}$}
\]
is the category whose objects are the bifibrant (i.e. both fibrant and cofibrant) objects of $\mathcal{C}$ and whose morphisms are the morphisms in $\mathcal{C}$ modulo the homotopy relation. A morphism in a model category $\mathcal{C}$ is a \emph{homotopy equivalence} if it descends to an isomorphism in $\on{Ho} \mathcal{C}$. The following result is a variant of the classical Whitehead theorem (cf. \cite[Ch 2, Thm 1.10]{goerss2009simplicial}).

\begin{theorem}[Whitehead] \label{thm:model_whitehead} A morphism $A \to B$ of bifibrant objects in a model category $\mathcal{C}$ is a weak equivalence if and only if it is a homotopy equivalence.
\end{theorem}

Later in this paper, we will briefly need to use derived functors, and specifically sequential homotopy colimits. We review the latter briefly here. Let $\mathcal{S}$ denote the ordered set $\N$ viewed as a small category. The category of functors
\[[\mathcal{S},\mathcal{C}]\qquad\text{in a model category $\mathcal{C}$}\]
consists of all sequential diagrams $A_1 \to A_2 \to \dots $ in $\mathcal{C}$. This category has a natural Reedy model structure \cite{riehl2014theory} and thus an associated homotopy category. The homotopy colimit is then a functor
\[
\on{hocolim}: \on{Ho}[\mathcal{S},\mathcal{C}] \to \on{Ho} \mathcal{C}
\]
See Riehl \cite[Ex 8.5]{riehl2014theory} for more details. Given a cofibrant diagram $F$ in $[\mathcal{S},\mathcal{C}]$, the homotopy colimit coincides with the ordinary colimit. More generally, for any diagram $F$ in $[\mathcal{S},\mathcal{C}]$, we can define the homotopy colimit by taking a cofibrant replacement $E \to F$ and taking
\[
\on{hocolim} F := \on{colim} E \in \on{Ho} \mathcal{C}
\]
In particular, there is a natural map from the homotopy colimit to the ordinary colimit. In the category of $\Z$-graded cdg-algebras over $R$, we can take the homotopy colimit to be a Sullivan dg-algebra by Lemma \ref{lem:cofibrant_replacement}.

\subsection{Model Structure On CDGA} The category of graded-commutative dg-algebras graded by $\Z$ is a classic example of a model category. Precisely, we have the following result of Hinich \cite{hinich1997homological}. 

\begin{theorem}[Hinich] \label{thm:model_structure_on_cdga} The category of $\Z$-graded commutative dg-algebras over a unital ring $R$
\[
\CDGA(R)
\]
carries a model category structure where weak equivalences are quasi-isomorphisms, fibrations are surjective maps and cofibrant objects are retracts of Sullivan dg-algebras in the sense of Definition \ref{def:Sullivan}.
\end{theorem} 

\begin{proof} The description of weak equivalences and fibrations is directly from \cite[\S 2.2]{hinich1997homological} and \cite[Thm 2.2.1]{hinich1997homological} applied to the forgetful functor $\CDGA(R) \to {\bf Ch}(R)$ to the category of $\Z$-graded chain complexes. Moreover, cofibrations are retracts of standard cofibrations $R \to A$ by \cite[Rmk 2.2.5]{hinich1997homological}. A standard cofibration $R \to A$ is precisely a Sullivan cdg-algebra by construction \cite[\S 2.2.3]{hinich1997homological}.\end{proof}

\begin{remark} The analogue of Theorem \ref{thm:model_structure_on_cdga} holds for the category $\CDGA_\star(R)$ of pointed cdg-algebras.
\end{remark}

\noindent In light of Theorems \ref{thm:model_whitehead} and \ref{thm:model_structure_on_cdga}, the main homological lemma (Lemma \ref{lem:main_homological}) is simply a version of Whitehead's Theorem in our setting. While this can be turned into a proof in the $\Z$-graded setting, for the general case we have opted to provide an elementary proof following Allday \cite{allday1978rational}.

\vspace{3pt}

In Section \ref{sec:example_applications}, we will need several further results about the model category of cdg-algebras. We start by introducing the following terminology.

\begin{definition} A cdg-algebra $A$ graded by $\Z$ is \emph{$k$-positively generated} for $k \ge 0$ if $A$ is generated by elements of grading larger than $k$ as an unital algebra. A $0$-positively generated pointed cdg-algebra will simply be called \emph{positively generated}.
\end{definition}

The following lemma is a strong variant of the usual factorization axiom in a model category.

\begin{lemma}[Sullivan Factorization] \label{lem:sullivan_factorization} Let $A$ and $B$ be pointed cdg-algebras with $A = SU$ Sullivan. Then any pointed cdg-algebra map $\Phi:A \to B$ factorizes as a composition of
\[
\text{a cofibration }A = SU \to SV \qquad\text{and}\qquad \text{a fibration and weak equivalence }SV \to B
\]
where $SV$ is Sullivan, and $SU \to SV$ is induced by an inclusion $U \to V$ Moreover, if $A$ and $B$ are $k$-positively generated then we can choose $SV$ to be $k$-positively generated.\end{lemma}

\begin{proof} We adapt an argument of Hinich \cite{hinich1997homological}. Specifically, we inductively constructing a sequence of Sullivan dg-algebras $SV_i$ where $V_{i+1} = V_i \oplus M_i$ and $V_0 = U$ along with a sequence of maps $\Phi_i:SV_i \to B$ such that the following diagram commutes.
\[
\begin{tikzcd}
SU \arrow{r}{\subset} \arrow{d}{\Phi} & SV_1 \arrow{r}{\subset} \arrow{d}{\Phi_1} & SV_2 \arrow{r}{\subset} \arrow{d}{\Phi_2} & \dots \\
B \arrow{r}{=} & B \arrow{r}{=} & B \arrow{r}{=} & \dots 
\end{tikzcd}
\]
Given a set $S$, we denote the free $R$-module generated by $S$ by $FS$. For the base case, we let $M_1$ be the free $R$-module with two generators $x_b$ and $y_b$ for each $b$ in the kernel $B_\star$ of the augmentation of $B$ and a generator $w_z$ for each element $z$ in the set of cycles $Z_\star \subset B_\star$. That is
\[
M_1 = FB_\star \oplus FB_\star \oplus FZ_\star
\]
We define the gradings by $|x_b| = |y_b| - 1 = |b|$ and $|w_z| = |z|$ and we extend the augmentation uniquely so that $M_0$ is in the kernel. We extend the differential on $SU$ to $SV_1 = S(U \oplus M_0)$ by setting $dx_b = y_b$ and $dw_z = 0$. We extend the map $\Phi$ to $\Phi_1$ by taking
\[\Phi_1(x_b) = b \qquad \Phi_1(y_b) = db \qquad \Phi_1(w_z) = z\]
For the induction step, let $Z_i \subset SV_i$ denote the set of cycles in $SV_i$ that are in the kernel of the augmentation and let $P_i$ denote the set of pairs $p = (z,b) \in Z_i \times B_\star$ where $\Phi_i(z) = db$. Set
\[
M_i = FP_i \qquad\text{with a generator $x_p$ for each $p \in P_i$}
\]
We define $|x_p| = |z|$ and take $M_i$ to be in the kernel of the extended augmentation. We also extend the differential to $SV_{i+1}$ by taking $dx_p = z \in SV_i$ and extend $\Phi_i$ to $\Phi_{i+1}$ by setting $\Phi_{i+1}(x_p) = b$.

\vspace{3pt}

Now note that each map $\Phi_i:SV_i \to B$ is surjective and maps cycles surjectively onto cycles. Moreover, if $z$ is a cycle in $SV_i$ such that $\Phi_i(z) \in B$ is a boundary, then $z$ is mapped to a boundary under the map $SV_i \to SV_{i+1}$. Thus the colimit
\[
\on{colim} \Phi_i :\on{colim} SV_i \to B
\]
is a surjective quasi-isomorphism, i.e. a fibration and a weak equivalence. It is simple to see that the colimit $SV$ is Sullivan and generated by $V = \on{colim} V_i$ and $SU \to SV$ is induced by the inclusion $U \subset V$. This gives the desired factorization. 

\vspace{3pt}

Finally, suppose that $A$ and $B$ are $k$-positively generated. It is straightforward to check that the generators of $M_1$ have grading larger than $k$, so that $SV_1$ is $k$-positively generated. Similarly, the generators $M_i$ introduced at each stage of the induction have grading larger than $k$ if $SV_i$ is $k$-positively generated. Thus each $SV_i$ and the colimit $SV$ are $k$-positively generated. \end{proof}

\begin{remark} \label{rmk:unpointed_factorization} Lemma \ref{lem:sullivan_factorization} also holds for maps $A \to B$ of (unpointed) cdg-algebras (without the final fact about $k$-positively generation). This is basically the original statement of Hinich \cite[\S 2.2.4]{hinich1997homological}. \end{remark}

\begin{lemma}[Sullivan Replacement] \label{lem:cofibrant_replacement} Any cdg-algebra $B$ has an acyclic fibration (i.e. surjective quasi-isomorphism) $A \to B$ from a Sullivan cdg-algebra $A$. Moreover, $SV$ is unique up to quasi-isomorphism. 
\end{lemma}

\begin{proof} Simply apply the factorization in Lemma \ref{lem:sullivan_factorization} for cdg-algebras without augmentations (Remark \ref{rmk:unpointed_factorization}) to the map $R \to B$. This yields an acyclic cofibration $A \to B$ from a Sullivan cdg-algebra. If $A$ and $A'$ are two such Sullivan cdg-algebras, then there is a diagram
\[
\begin{tikzcd}
R \arrow[r] \arrow[d] & A \arrow[d,"\simeq"] \\
A'  \arrow[r,"\simeq"] & B
\end{tikzcd}
\]
Since the left arrow is a cofibration, and the right arrow is a fibration and a weak equivalence, there is a map $A' \to A$ making the diagram commute by the lifting axiom \cite[Def 2.1]{hinich1997homological}. This map must necessarily be a weak equivalence. 
\end{proof}

\begin{lemma}[Positively Generated Hocolim] \label{lem:connected_hocolim} Let $A_1 \to A_2 \to \dots$ be a sequence of $\Z$-graded Sullivan cdg-algebras $A_i$ that admit an augmentation and are each $k$-positively generated for $k \ge 0$. Then
\[
\on{hocolim} \; A_i \qquad\text{is quasi-isomorphic to a $k$-positively generated Sullivan cdg-algebra}
\]
\end{lemma}

\begin{proof} Write $A_i = SV_i$ where $V_i$ is generated by elements of degree larger than $k$. There is a unique graded algebra map $\epsilon_i:SV_i \to R$ (determined by the fact that $V_i$ maps to zero) and this map must be the augmentation of $SV_i$. Moreover, the maps $SV_i \to SV_{i+1}$ must be pointed with respect to $\epsilon_i$ and $\epsilon_{i+1}$ by grading considerations. Thus we may view the sequence
\[
A_1 \to A_2 \to A_3 \to \dots
\]
canonically as a diagram of pointed cdg-algebra. We now construct a commutative diagram
\begin{equation}
\begin{tikzcd}
SU_1 \arrow{r}{} \arrow{d}{\simeq} & SU_2 \arrow{r}{} \arrow{d}{\simeq} & SU_3 \arrow{r}{} \arrow{d}{\simeq} & \dots \\
A_1 \arrow{r}{} & A_2 \arrow{r}{} & A_3 \arrow{r}{} & \dots 
\end{tikzcd}
\end{equation}
where $SU_i$ is Sullivan and $k$-positively generated, and $SU_i \to SU_{i+1}$ is a cofibration induced by an inclusion $U_i \to U_{i+1}$. Then $SU = \on{colim} SU_i$ is $k$-positively generated Sullivan and this is the homotopy colimit by Riehl \cite[Ex 8.5]{riehl2014theory}. 

\vspace{3pt}

We construct the diagram by induction on $i$. For the base case, take $U_1 = V_1$ and $SU_1 \to SV_1 = A_1$ to be the identity. For the induction step, suppose that $SU_i$ and the maps $SU_{j-1} \to SU_j$ and $SU_j \to A_j$ have been defined up to $j = i$. By Lemma \ref{lem:sullivan_factorization}, there is a factorization of the composition map $SU_i \to A_i \to A_{i+1}$ as
\[
SU_i \to SU_{i+1} \to A_{i+1}
\]
where the map $SU_i \to SU_{i+1}$ is given by inclusion $U_i \to U_{i+1}$ and $SU_{i+1} \to A_{i+1}$ is a quasi-isomorphism. Moreover, by induction $SU_i$ is $k$-positively generated and $A_{i+1}$ is also, by assumption. Thus $SU_{i+1}$ is $k$-positively generated. \end{proof}

\section{Uniqueness Of Augmentations} \label{sec:example_applications}

Several properties and applications of linearized contact homology are accessible using Theorem \ref{thm:main_theorem} and the algebraic results of Section \ref{sec:model_categories}, despite the shortcomings noted in Remarks \ref{rmk:fillings}-\ref{rmk:weak_functoriality}. 

\vspace{3pt}

As a simple example, we discuss the uniqueness of augmentations and linearized contact homology for strongly asymptotically dynamically convex (SADC) contact forms in the sense of Zhou \cite{zhou2021symplectic,zhou2022symplectic}. These statements are known folklore that have partially motivated various rigidity statements about symplectic homology and fillings (cf. \cite[Rmk 1.7]{zhou2021symplectic} or \cite[p. 4]{zhou2022symplectic}). Here we give rigorous formulations and proofs of these statements.

\subsection{ADNH Contact Manifolds} We start by discussing a class of contact manifolds (including SADC manifolds) whose contact dg-algebra admits an enhanced grading.

\begin{definition} An \emph{admissible pair} $(\alpha,L)$ for a contact manifold $(Y,\xi)$ is a pair of a contact form $\alpha$ and an action $L > 0$ such that every Reeb orbit $\Gamma$ of action less than or equal to $L$ is non-degenerate.
\end{definition}

Given an admissible pair $(\alpha,L)$ on $(Y,\xi)$, we may define a version of the contact dg-algebra generated only by Reeb orbits of $\alpha$ with action $L$ or less.
\[
A^L(Y,\alpha) = SV_L(Y,\alpha) \qquad\text{ as in Section \ref{sec:contact_homology}}
\]
The cobordism maps of Pardon \cite{pardon2019contact} then restrict to dg-algebra maps of the following form.
\[
A^L(Y,\alpha) \to A^K(Y,\beta) \qquad\text{if}\qquad \alpha/L > \beta/K
\]
Here these dg-algebra maps depend on a choice of cobordism Floer data (cf. \cite[\S 1.2]{pardon2019contact}) but any two such maps are chain homotopic. Note that the set of admissible pairs forms a directed poset under the ordering above.  
\begin{lemma} \label{lem:colimit_admissible_CH}The contact homology of $(Y,\xi)$ is given by the following colimit over admissible pairs.
\[
CH(Y,\xi) = \underset{(\alpha,L)}{\on{colim}} \; H(A^L(Y,\alpha))
\]
\end{lemma}

\begin{proof} Fix a non-degenerate contact form $\alpha$ and a sequence $L_i$ of actions with $L_i \to \infty$. Note that $A(Y,\alpha)$ is the colimit of $A^{L_i}(Y,\alpha)$ in the category of chain complexes, where homology commutes with filtered colimits. Thus we see that
\[\on{colim}_i \; H(A^{L_i}(Y,\alpha)) = H\Big(\on{colim}_i \; A^{L_i}(Y,\alpha)\Big) = H(A(Y,\alpha)) = CH(Y,\xi)\]
The lemma follows since $(\alpha,L_i)$ is cofinal in the set of admissible pairs.\end{proof}

\begin{definition} \label{def:ADNH} A closed contact manifold $(Y,\xi)$ is \emph{asymptotically dynamically null-homologous (ADNH)} if there is a cofinal sequence $(\alpha_i,L_i)$ of admissible pairs such that
\[[\Gamma] = 0 \in H_1(Y)\]
for every closed orbit $\Gamma$ of $\alpha_i$ with period less than or equal to $L_i$.
\end{definition}

\begin{lemma} \label{lem:ADNH_homology} Let $(Y,\xi)$ be an ADNH contact manifold. Then the inclusion
\[
A_0(Y,\xi) \to A(Y,\xi) \qquad\text{is a quasi-isomorphism}
\]
\end{lemma}

\begin{proof} The inclusion map $CH_0(Y,\xi) \to CH(Y,\xi)$ is simply a colimit over all admissible pairs $(\alpha,L)$ of the following inclusion maps
\[H(A^L_0(Y,\alpha)) \to H(A^L(Y,\alpha)) \qquad\text{where}\qquad A^L_0(Y,\alpha) = A^L(Y,\alpha) \cap A_0(Y,\alpha)\]
If $(Y,\xi)$ is ADNH, then these maps are isomorphisms for a cofinal sequence of pairs.
\end{proof}

\begin{lemma} \label{lem:ADNH_lift_to_Z} Let $(Y,\xi)$ be an ADNH contact manifold with $c_1(\xi) = 0$. Then the contact dg-algebra canonically lifts to an quasi-isomorphism class of $\Z$-graded Sullivan dg-algebra $C(Y,\xi)$. Moreover
\[
C(Y,\xi) \simeq \on{hocolim} A^{L_i}(Y,\alpha_i) \qquad\text{for any cofinal sequence $(\alpha_i,L_i)$ as in Definition \ref{def:ADNH}}
\]
\end{lemma}

\begin{proof} We simply take the quasi-isomorphism class of $\Z$-graded Sullivan cdg-algebra to be the equivalence class of any Sullivan cofibrant replacement of $A_0(Y,\xi)$.
\[
C(Y,\xi) \qquad\text{with a quasi-isomorphism}\qquad C(Y,\xi) \xrightarrow{\simeq} A_0(Y,\xi)
\]
This replacement exists and is well-defined by Lemma \ref{lem:cofibrant_replacement}. For the second claim, let $(\alpha_i,L_i)$ be as in Definition \ref{def:ADNH}. By scaling each $\alpha_i$ we may assume that $\alpha_i > \alpha_{i+1}$ and $L_i < L_{i+1}$ with $L_i \to \infty$. Fix a non-degenerate contact form $\alpha > \alpha_1$ and consider the commutative diagram
 \begin{equation} \label{eq:two_rows}
\begin{tikzcd}
A^{L_1}(Y,\alpha) \arrow{r}{} \arrow{d}{} & A^{L_2}(Y,\alpha) \arrow{r}{} \arrow{d}{} & A^{L_3}(Y,\alpha) \arrow{r}{} \arrow{d}{} & A^{L_4}(Y,\alpha) \arrow{r}{} \arrow{d}{} & \dots \\
A^{L_1}(Y,\alpha_1) \arrow{r}{} & A^{L_2}(Y,\alpha_2)  \arrow{r}{} & A^{L_3}(Y,\alpha_3)  \arrow{r}{} & A^{L_4}(Y,\alpha_4) \arrow{r}{} & \dots
\end{tikzcd}
\end{equation}
Here the top row is the action filtration on $A(Y,\alpha)$, the bottom row is the diagram in the lemma statement and the vertical maps are the composition maps
\[
A^{L_i}(Y,\alpha) \to A^{L_i}(Y,\alpha_1) \to \dots \to A^{L_i}(Y,\alpha_i)
\]
By taking homotopy colimits between the diagram (\ref{eq:two_rows}) we get a sequence of maps
\[A_0(Y,\alpha) \subset A(Y,\alpha) = \on{colim} \;A^{L_i}(Y,\alpha) \to \on{hocolim}\; A^{L_i}(Y,\alpha_i) \]
This induces an isomorphism on homology by Lemma \ref{lem:colimit_admissible_CH}, Lemma \ref{lem:ADNH_homology} and the fact that homotopy colimits commute with taking homology.
\end{proof}

\noindent We write the corresponding set of $\Z$-graded augmentation classes of this lift by
\[\on{Aug}_\Z(Y,\xi) := \on{Aug}(C(Y,\xi))\]




\subsection{SADC Contact Manifolds} We next discuss strong asymptotic dynamical convexity as formulated by Zhou \cite{zhou2021symplectic,zhou2022symplectic}. This is a variant of the asymptotic dynamical convexity of Lazarev \cite{lazarev2020contact}, which is a generalization of index posivity (cf. Cieliebak-Oancea \cite{cieliebak2018symplectic}).

\begin{definition} \label{def:SADC} A closed contact manifold $(Y,\xi)$ is \emph{$k$ strongly asymptotically dynamically convex} (\emph{$k$-SADC}) if  $c_1(\xi) = 0$ and there is a sequence $(\alpha_i,L_i)$ of a contact forms and actions with
\[
\alpha_i > \alpha_{i+1} \qquad L_i < L_{i+1} \quad\text{and}\quad L_i \to \infty
\]
and with the property that every closed Reeb orbit $\Gamma$ of $\alpha_i$ with period $L_i$ or less satisfies
\[\Gamma \text{ non-degenerate}\qquad \Gamma \text{ contractible} \qquad\text{and}\qquad |\Gamma| = n - 3 + \CZ(\Gamma) > k\]
Here $Y$ is dimension $2n-1$ and the grading $|\Gamma|$ is over $\Z$ (due to the hypotheses). A $0$-SADC contact manifold will simply be called \emph{SADC}. \end{definition}

\begin{example} \cite[Ex 3.7]{zhou2021symplectic} The standard contact sphere is SADC in dimension three and up. In fact,  the SADC property is preserved by sub-critical surgery preserving the zero Chern class condition by Lazarev \cite{lazarev2020contact} and so all sub-critically fillable contact manifolds with zero Chern class are SADC in dimension three and up. 
\end{example}

Any strongly asymptotically dynamically convex contact manifold is also asymptotically dynamically null-homologous with vanishing Chern class. Thus by Lemma \ref{lem:ADNH_lift_to_Z}, we can view $A(Y,\xi)$ as a $\Z$-graded Sullivan dg-algebra and consider the set of $\Z$-graded augmentation classes.

\begin{proposition} \label{prop:SADC_unique_aug} Let $(Y,\xi)$ be a $k$-SADC contact manifold (that admits a $\Z$-graded augmentation if $k \ge 0$). Then $A(Y,\xi)$ is quasi-isomorphic to a $k$-positively generated cdg-algebra.
\end{proposition}

\begin{proof} Fix a sequence of pairs $(\alpha_i,L_i)$ as in Definition \ref{def:SADC} with $\alpha_i$ non-degenerate and a contact form $\alpha$ with $\alpha > \alpha_1$. By Lemma \ref{lem:ADNH_lift_to_Z}, we have a quasi-isomorphism
\[
A(Y,\xi) \simeq \on{colim} A^{L_i}(Y,\alpha_i)
\]
By Definition \ref{def:SADC}, each cdg-algebra $A^{L_i}(Y,\alpha_i)$ is $k$-positively generated. Moreover, each cdg-algebra $A^{L_i}(Y,\alpha_i)$ must admit a $\Z$-graded augmentation. If $k \ge 1$, this is automatic since the trivial projection to $\Q$ is an augmentation. If $k = 1$, then this follows since $A^{L_i}(Y,\alpha_i)$ maps to $A(Y,\alpha_i) \simeq A(Y,\xi)$, which has an augmentation. Thus we may apply Lemma \ref{lem:connected_hocolim} to find that
\[
\on{hocolim} A^{L_i}(Y,\alpha_i) \qquad\text{is $k$-positively generated up to quasi-isomorphism} \qedhere
\]
\end{proof}

\begin{corollary} \label{cor:SADC_unique_aug_1} Let $(Y,\xi)$ be an SADC contact manifold. Then $(Y,\xi)$ admits at most one $\Z$-graded augmentation up to weak equivalence or up to homotopy of augmentations.
\end{corollary}

\begin{corollary} \label{cor:SADC_unique_aug_2} Let $(Y,\xi)$ be a $1$-SADC contact manifold. Then $(Y,\xi)$ admits a unique $\Z$-graded augmentation up to weak equivalence or up to homotopy of augmentations.
\end{corollary}

We conclude this note with a few corollaries of Proposition \ref{prop:SADC_unique_aug} and Corollaries \ref{cor:SADC_unique_aug_1}-\ref{cor:SADC_unique_aug_2}. First we note the following obvious consequences for linearized contact homology.

\begin{corollary} The $\Z$-graded linearized contact homology of $(Y,\xi)$ is independent of the $\Z$-graded augmentation up to non-canonical isomorphism if $(Y,\xi)$ is SADC. In this case, we write
\[
LCH(Y,\xi) = LCH_{[\epsilon]}(Y,\xi) \qquad\text{for the unique augmentation class $[\epsilon]$}
\]
\end{corollary}

\begin{corollary} Given an exact filling $(W,\lambda)$ of an SADC contact manifold $(Y,\xi)$ with $c_1(W) = 0$, the linearized homology with respect to the augmentation $\epsilon_W$ induced by $W$
\[
LCH(W) = LH(A(Y,\alpha),\epsilon_W)
\]
is independent of the filling $W$ and all choices of Floer data, up to non-canonical isomorphism.\end{corollary}

Finally, recent work of Avdek \cite{avdek2023algebraic} gave a beautiful computation of the contact homology of a convex sutured neighborhood of a convex hypersurface. Recall (cf. \cite{honda2019convex,chaidez2024robustly,avdek2023algebraic}) that given an $\R$-invariant contact structure $\xi$ on $\R \times \Sigma$, there is a natural contact sub-manifold $\Gamma \subset \Sigma$ called the dividing set, that divides $\Sigma$ into two ideal Liouville domains $\Sigma_\pm$ filling $\Gamma$. These determine augmentations $\epsilon_\pm$ of the contact dg-algebra of $\Gamma$. In \cite{avdek2023algebraic}, Avdek proved the following result.

\begin{theorem} \cite[Thm 1.1.1]{avdek2023algebraic} \label{thm:avdek} Let $\xi$ be an $\R$-invariant contact structure on $U = \R \times \Sigma$. Let $\Sigma = 0 \times \Sigma$ be the corresponding convex surface with dividing set $\Gamma$. Then
\[CH(U,\xi) = S\Big(\widehat{LCH}_{[\epsilon_+]}(\Gamma)\Big) \text{ if $\epsilon_+$ and $\epsilon_-$ are homotopic} \qquad\text{and}\qquad CH(U,\xi) = 0 \text{ otherwise.}\]
\end{theorem}

\noindent Here $\widehat{LCH}$ denotes the linearized contact homology with a grading shift of $+1$. Note that the augmentations in Theorem \ref{thm:avdek} must be homotopic in the sense of cdg-algebras (cf. Avdek \cite[Section 14.2]{avdek2023algebraic}). The following is a corollary of Proposition \ref{prop:SADC_unique_aug} and Theorem \ref{thm:avdek}.

\begin{corollary} \label{cor:convex_surfaces} Let $\Sigma$ be a closed convex hypersurface of $(Y,\xi)$ with SADC convex dividing set $\Gamma \subset \Sigma$. Let $U$ be a convex sutured neighborhood of $\Sigma$ and suppose that $c_1(\xi|_U) = 0 \in H^2(U)$. Then
\[
CH(U,\xi|_U) \simeq S\big(\widehat{LCH}(\Gamma)\big)
\]
In particular, the contact homology of $U$ depends only on the dividing set $\Gamma$ in this case. \end{corollary}

\begin{proof} One may check that $c_1(\xi|_U) = 0$ implies that $c_1(\Sigma_\pm) = 0$ so that the exact fillings $\Sigma_+$ and $\Sigma_-$ induce $\Z$-graded augmentations $\epsilon_+$ and $\epsilon_-$ on $\Gamma$. By Corollary \ref{cor:SADC_unique_aug_1}
\[
LCH_{[\epsilon_+]}(\Gamma) = LCH(\Gamma) \text{ is independent of $\epsilon_+$}
\]
Thus by Theorem \ref{thm:avdek}, it suffices to show that $CH(U,\xi|_U) \neq 0$. In \cite[Thm 1.2.1]{avdek2023algebraic} Avdek constructs for any cofinal sequence $(\alpha_i,L_i)$ of admissible pairs for the dividing set $\Gamma$, a corresponding cofinal sequence of cofinal pairs $(\beta_i,L_i)$ for $U$ such that there is a bijection $\gamma \mapsto \hat{\gamma}$ from Reeb orbits of $\Gamma$ to Reeb orbits of $U$ with $|\hat{\gamma}| = |\gamma| + 1$. This implies that if $\Gamma$ is $0$-SADC, then $U$ is $1$-SADC. Then by Corollary \ref{cor:SADC_unique_aug_2}, $A(U,\xi|_U)$ must have an augmentation, and so cannot have zero homology.  \end{proof}

\subsection*{Acknowledgements.} We would like to thank the anonymous referee for various helpful comments, including the suggestion for a more detailed proof of Corollary \ref{cor:convex_surfaces}.

\bibliographystyle{hplain}
\bibliography{standard_bib}

\begin{thebibliography}{10}

\bibitem{allday1978rational}
Christopher Allday.
\newblock On the rational homotopy of fixed point sets of torus actions.
\newblock {\em Topology}, 17(1):95--100, 1978.

\bibitem{avdek2023algebraic}
Russell Avdek.
\newblock An algebraic generalization of giroux's criterion.
\newblock {\em arXiv preprint arXiv:2307.09068}, 2023.

\bibitem{bao2023semi}
Erkao Bao and Ko~Honda.
\newblock Semi-global kuranishi charts and the definition of contact homology.
\newblock {\em Advances in Mathematics}, 414:108864, 2023.

\bibitem{bourgeois2003compactness}
Fr{\'e}d{\'e}ric Bourgeois, Yakov Eliashberg, Helmut Hofer, Kris Wysocki, and Eduard Zehnder.
\newblock Compactness results in symplectic field theory.
\newblock {\em Geometry \& Topology}, 7(2):799--888, 2003.

\bibitem{bourgeois2017s}
Fr{\'e}d{\'e}ric Bourgeois and Alexandru Oancea.
\newblock S 1-equivariant symplectic homology and linearized contact homology.
\newblock {\em International Mathematics Research Notices}, 2017(13):3849--3937, 2017.

\bibitem{chaidez2024robustly}
Julian Chaidez.
\newblock Robustly non-convex hypersurfaces in contact manifolds.
\newblock {\em arXiv:2406.05979}, 2024.

\bibitem{cieliebak2018symplectic}
Kai Cieliebak and Alexandru Oancea.
\newblock Symplectic homology and the eilenberg--steenrod axioms.
\newblock {\em Algebraic \& Geometric Topology}, 18(4):1953--2130, 2018.

\bibitem{colin2013reeb}
Vincent Colin and Ko~Honda.
\newblock Reeb vector fields and open book decompositions.
\newblock {\em Journal of the European Mathematical Society}, 15(2):443--507, 2013.

\bibitem{eliashberg2010introduction}
Yakov Eliashberg, A~Glvental, and Helmut Hofer.
\newblock Introduction to symplectic field theory.
\newblock {\em Visions in Mathematics: GAFA 2000 Special Volume, Part II}, pages 560--673, 2010.

\bibitem{fish2018lectures}
Joel~W Fish and Helmut Hofer.
\newblock Lectures on polyfolds and symplectic field theory.
\newblock {\em arXiv:1808.07147}, 2018.

\bibitem{goerss2009simplicial}
Paul~G Goerss and John~F Jardine.
\newblock {\em Simplicial homotopy theory}.
\newblock Springer Science \& Business Media, 2009.

\bibitem{hind2024symplectic}
Richard Hind and Kyler Siegel.
\newblock Symplectic field theory: an overview.
\newblock {\em arXiv preprint arXiv:2410.19936}, 2024.

\bibitem{hinich1997homological}
Vladimir Hinich.
\newblock Homological algebra of homotopy algebras.
\newblock {\em Communications in algebra}, 25(10):3291--3323, 1997.

\bibitem{honda2019convex}
Ko~Honda and Yang Huang.
\newblock Convex hypersurface theory in contact topology.
\newblock {\em arXiv:1907.06025}, 2019.

\bibitem{hovey2007model}
Mark Hovey.
\newblock {\em Model categories}.
\newblock Number~63. American Mathematical Soc., 2007.

\bibitem{lazarev2020contact}
Oleg Lazarev.
\newblock Contact manifolds with flexible fillings.
\newblock {\em Geometric and Functional Analysis}, 30:188--254, 2020.

\bibitem{lurie2009higher}
Jacob Lurie.
\newblock {\em Higher topos theory}.
\newblock Princeton University Press, 2009.

\bibitem{moreno2020landscape}
Agustin Moreno and Zhengyi Zhou.
\newblock A landscape of contact manifolds via rational sft.
\newblock {\em arXiv:2012.04182}, 2020.

\bibitem{pardon2016algebraic}
John Pardon.
\newblock An algebraic approach to virtual fundamental cycles on moduli spaces of pseudo-holomorphic curves.
\newblock {\em Geometry \& Topology}, 20(2):779--1034, 2016.

\bibitem{pardon2019contact}
John Pardon.
\newblock Contact homology and virtual fundamental cycles.
\newblock {\em Journal of the American Mathematical Society}, 32(3):825--919, 2019.

\bibitem{quillen2006homotopical}
Daniel~G Quillen.
\newblock {\em Homotopical algebra}.
\newblock Springer-Verlag, 1967.

\bibitem{riehl2014theory}
Emily Riehl and Dominic Verity.
\newblock The theory and practice of reedy categories.
\newblock {\em Theory and Applications of Categories}, 29:256--301, 2014.

\bibitem{zhou2021symplectic}
Zhengyi Zhou.
\newblock Symplectic fillings of asymptotically dynamically convex manifolds i.
\newblock {\em Journal of Topology}, 14(1):112--182, 2021.

\bibitem{zhou2022symplectic}
Zhengyi Zhou.
\newblock Symplectic fillings of asymptotically dynamically convex manifolds ii--k-dilations.
\newblock {\em Advances in Mathematics}, 406:108522, 2022.

\end{thebibliography}

\end{document}